\documentclass[preprint,12pt]{elsarticle}

\usepackage{amsthm}
\usepackage{amsmath,amssymb,txfonts}
\usepackage{color,bm,comment}

\newtheorem{theorem}{Theorem}
\newtheorem{lemma}{Lemma}
\newtheorem{prop}{Proposition}
\newtheorem{cor}{Corollary}
\newtheorem{remark}{Remark}

\theoremstyle{definition}

\def\re{\mathbb{R}}
\def\N{\mathbb{N}}

\def\({\left(}
\def\){\right)}
\def\[{\left[}
\def\]{\right]}
\def\pd{\partial}
\def\lap{\Delta}

\def\ep{\varepsilon}
\def\w{\omega}
\def\la{\lambda}

\begin{document}

\begin{frontmatter}



\title{Explicit optimal constants of two critical Rellich inequalities for radially symmetric functions}


\author[M]{Megumi Sano}
\ead{smegumi@hiroshima-u.ac.jp}
\address[M]{Laboratory of Mathematics, Graduate School of Engineering,
Hiroshima University, Higashi-Hiroshima, 739-8527, Japan}


\date{\today}

\begin{abstract}
We consider two critical Rellich inequalities with singularities at both the origin and the boundary in the higher order critical radial Sobolev spaces $W_{0, {\rm rad}}^{k, p}$, where $1< p = \frac{N}{k}$. 
We give the explicit values of the optimal constants of two critical Rellich inequalities for radially symmetric functions in $W_{0, {\rm rad}}^{k, \frac{N}{k}}$.
Furthermore the (non-)attainability of the optimal constants are also discussed. 
\end{abstract}

\end{frontmatter}



%
%
\section{Introduction}

Let $N \ge 2, 1 < p < N$ and $B_R \subset \re^N$ be the ball with the center $0$ and the radius $R >0$. The classical Hardy inequality:
\begin{equation}
\label{H_p}
\( \frac{N-p}{p} \)^p \int_{B_R} \frac{|u|^p}{|x|^p} dx \le \int_{B_R} | \nabla u |^p dx
\end{equation}
holds for all $u \in W^{1,p}_0(B_R)$, where $W_0^{1,p}(B_R)$ is a completion of $C_c^{\infty}(B_R)$ with respect to the norm $\| \nabla (\cdot )\|_{L^p(B_R)}$. We refer the celebrated work by G. H. Hardy \cite{Hardy(1919)}. 
The inequality (\ref{H_p}) has great applications to partial differential equations, for example stability, global existence and instantaneous blow-up and so on. See e.g. \cite{BV}, \cite{BG}. 
It is well-known that $(\frac{N-p}{p})^p$ in (\ref{H_p}) is the optimal constant, $(\frac{N-p}{p})^p$ is not attained in $W_0^{1,p}(B_R)$ and (\ref{H_p}) expresses the embedding $W_0^{1,p} \hookrightarrow L^{p^*, p}$, where $p^* = \frac{Np}{N-p}$ is the Sobolev critical exponent and $L^{p, q}$ are the Lorentz spaces. 
By an inclusion property of the Lorentz spaces: $L^{p^*, p} \hookrightarrow L^{p^*, q}$ for any $q \in [p, \infty]$, the Hardy inequality (\ref{H_p}) is stronger than the Sobolev inequality: $S_{N,p} \| u\|_{p^*} \le \| \nabla u\|_p$ in the view of the embedding as follows.
$$
W_0^{1,p} \hookrightarrow L^{p^*, p} \hookrightarrow L^{p^*, p^*} = L^{p^*}
$$
It is also well-known that Sobolev's optimal constant $S_{N,p}$ is not attained on $B_R$ and $S_{N,p}$ is attained on $\re^N$ (ref. \cite{Au, T}), 

%
%

On the other hand, in the critical case where $p=N$, 
both two inequalities become trivial inequalities since two optimal constants $(\frac{N-p}{p})^p, S_{N,p}$ become zero. Instead of these two inequalities, the Trudinger-Moser inequality is considered as a limiting case of the Sobolev inequality (ref. \cite{Tru, M}) and the critical Hardy inequality: 
\begin{align}\label{H_N}
\( \frac{N-1}{N} \)^N \int_{B_R} \dfrac{|u|^N}{|x|^N (\log \frac{aR}{|x|})^N} dx \le \int_{B_R} | \nabla u |^N dx \quad \( \forall u \in W^{1,N}_0(B_R),\, a \ge1 \,\)
\end{align}
is considered as a limiting case of the Hardy inequality (ref. \cite{L, La, BFT(IUMJ), BFT(TAMS), MOW(Tohoku), TF}). 
It is known that $(\frac{N-1}{N})^N$ in (\ref{H_N}) is the optimal constant, $(\frac{N-1}{N})^N$ is not attained in $W_0^{1,N}(B_R)$ (ref. \cite{ASand, AE, II} etc.) and (\ref{H_N}) expresses the embedding $W_0^{1,N} \hookrightarrow L^{\infty, N} (\log L)^{-1}$, where $L^{p, q}(\log L)^r$ are the Lorentz-Zygmund spaces. 
By an inclusion property of the Lorentz-Zygmund spaces: $L^{\infty, N} (\log L)^{-1} \hookrightarrow L^{\infty, q}(\log L)^{-1+ \frac{1}{N} -\frac{1}{q}}$ for any $q \in [N, \infty]$ (see e.g. \cite{BS} Theorem 9.5), the critical Hardy inequality (\ref{H_N}) is stronger than the Trudinger-Moser inequality in the view of the embedding as follows.
$$
W_0^{1,N} \hookrightarrow L^{\infty, N} (\log L)^{-1} \hookrightarrow L^{\infty, \infty}(\log L)^{-1+ \frac{1}{N}} = {\rm ExpL}^{\frac{N}{N-1}}
$$
It is also known that Trudinger-Moser's optimal constant is attained on $B_R$ (ref. \cite{CC}). For these embedding theorems, definitions and classification of these function spaces (the rearrangement invariant spaces, the Orlicz spaces and so on), see e.g. \S 1 in \cite{CF}. 

%
%

In the present paper, we focus on the higher order case. First, a higher order generalization of (\ref{H_p}) was proved by Rellich \cite{Rellich} in $W^{2,2}_0(B_R ), N \ge 5$. 
In general, let $k, m \in \N, m \ge 1, k \ge 2$ and $1<p< \frac{N}{k}$. Set $|u|_{k,p} = \| \nabla^k u \|_{L^p(B_R)}$,
\begin{align*}
&\nabla^k u = \begin{cases}
              \lap^m u  \quad &\text{if} \,\, k=2m,\\
              \nabla \lap^m u  &\text{if} \,\, k=2m+1,
              \end{cases}\\
&A_{k,p} = \begin{cases}
              \prod_{\ell =1}^{m} \dfrac{\{ N-2\ell p \} \{ N(p-1) +2(\ell -1)p \} }{p^2} \quad &\text{if} \,\, k=2m,\\
              \frac{N-p}{p} \prod_{\ell =1}^{m} \dfrac{\{ N-(2\ell +1)p \} \{ N(p-1) +(2\ell  -1)p \} }{p^2}  &\text{if} \,\, k=2m+1.
              \end{cases}
\end{align*}
Then the Rellich inequality:
\begin{equation}\label{Rel}
|u|^p_{k,p} \ge A_{k,p}^p \int_{B_R} \frac{|u|^p}{|x|^{kp}} dx
\end{equation}
holds for all $u \in W_0^{k,p}(B_R )$, where $W_0^{k,p}(B_R)$ is a completion of $C_c^{\infty}(B_R)$ with respect to $| (\cdot )|_{k,p}$. It is also known that $A_{k,p}^p$ are the optimal constants 
(ref. \cite{DH, Mitidieri, GGM, ST(Rellich)} ). For the higher order Sobolev inequalities, we refer \cite{R, Lions, Sw, CN}.

%
%

In the critical case where $p=\frac{N}{k}$, both two inequalities do not hold. Instead of the higher order Sobolev inequalities, the higher order Trudinger-Moser inequalities which are called the Adams inequalities are considered (ref. \cite{A}, see also \cite{RS}). Especially, Adams \cite{A} gave the explicit values of the optimal exponents of them in $W_0^{k,\frac{N}{k}}(B_R)$. On the other hand, instead of the Rellich inequalities (\ref{Rel}), the second order critical Rellich inequality:
\begin{equation}\label{R_2}
\int_{B_R} |\lap u|^{\frac{N}{2}} \,dx \ge \( \frac{N-2}{\sqrt{N}} \)^N \int_{B_R} \frac{|u|^{\frac{N}{2}}}{|x|^N \( \log \frac{aR}{|x|} \)^{\frac{N}{2}}} dx \quad \( \forall u \in W^{2,\frac{N}{2}}_0(B_R),\, a \gg 1 \)
\end{equation}
is considered (ref. \cite{DHA, AGS, ASant, CM(2012)}). 
It is known that the constant $\( \frac{N-2}{\sqrt{N}} \)^N$ in (\ref{R_2}) is optimal. 
In the higher order case where $W_{0,{\rm rad}}^{k,2}(B_R ) \,(k \ge 3)$, the validities of the critical Rellich inequalities and the optimal constants are treated in a pioneer work by \cite{ASant}. But, unfortunately, it seems that the argument in \cite{ASant} contains a gap, see Remark \ref{rem AS}. 
The purpose of this paper is to show the validities of the critical Rellich inequalities and give the optimal exponents and the optimal constants in $W_{0,{\rm rad}}^{k,\frac{N}{k}}(B_R )$ when $a=1$. 
Namely, we show the positivity of the optimal constants $R^{{\rm rad}}_{k, \gamma}$ and give the optimal exponents with respect to $\gamma$ and the explicit values of $R^{{\rm rad}}_{k, \gamma}$. Furthermore, we show the (non-)attainability of $R^{{\rm rad}}_{k, \gamma}$. 
Here $R_{k, \gamma}$ and $R^{{\rm rad}}_{k, \gamma}$ are given by
\begin{align*}
R_{k, \gamma} = \inf_{u \in W_0^{k,p}(B_R ) \setminus \{ 0\} } \dfrac{|u|^p_{k,p}}{ \int_{B_R} \frac{|u|^p}{|x|^N \( \log \frac{R}{|x|} \)^\gamma} \,dx}, 
\quad R^{{\rm rad}}_{k, \gamma} = \inf_{u \in W_{0,{\rm rad}}^{k,p}(B_R ) \setminus \{ 0\} } \dfrac{|u|^p_{k,p} }{ \int_{B_R} \frac{|u|^p}{|x|^N \( \log \frac{R}{|x|} \)^\gamma} \,dx}. 
\end{align*}
In the case where $a>1$, the potential function $|x|^{-N} \( \log \frac{aR}{|x|} \)^{-\gamma}$ has a singularity only at the origin. Thus there is an optimal exponent with respect to $\gamma$ in this case. However, in the case where $a=1$, the potential function has singularities not only at the origin but also at the boundary $\pd B_R$. Therefore there are two optimal exponents with respect to $\gamma$ in this case.

\noindent
Our main result is as follows.

\begin{theorem}\label{Main thm}
Let $N > k \ge 2, m \in \N$ and $p=\frac{N}{k}$. 
Then we have the followings.

\noindent
(i) $R_{k, \gamma}=R^{{\rm rad}}_{k, \gamma} = 0$ if $\gamma \not\in [p,N]$, and $R^{{\rm rad}}_{k, \gamma} > 0$ if $\gamma \in [p,N]$.\\
(ii) $R^{{\rm rad}}_{k, \gamma}$ is attained if $\gamma \in (p,N)$.\\
(iii) $R^{{\rm rad}}_{k, \gamma}$ is not attained if $\gamma = p, N$. 
Moreover we have
\begin{align*}
R^{{\rm rad}}_{k, p}= 
\begin{cases}
\( \frac{N-k}{kN} \prod_{j=1}^m 2j \,(N-2j)\, \)^p &\text{if}\,\, k=2m, \\
\( \frac{N-k}{N} \prod_{j=1}^m 2j \,(N-2j) \,\)^p &\text{if}\,\, k=2m+1,
\end{cases}\quad
R^{{\rm rad}}_{k, N}= \( \prod_{j=1}^k \frac{jN-k}{N} \)^p. 
\end{align*} 
\end{theorem}

Actually, the explicit value of $R^{{\rm rad}}_{k, p}$ in Theorem \ref{Main thm} (iii) is already founded by \cite{N} on the homogeneous groups.  

\begin{remark}
From the known results \cite{O, DHA, BT(2006)}, we see that $R_{2,\gamma} >0$ if $\gamma \in [p, N]$ since for $a >1$
\begin{align*}
\frac{1}{|x|^N \( \log \frac{R}{|x|} \)^p} \le \frac{2}{|x|^N \( \log \frac{aR}{|x|} \)^p}\,\,\text{holds near the origin and}\,\, \log \frac{R}{|x|} \ge \frac{{\rm dist} (x, \pd B_R)}{R}.
\end{align*} 
\end{remark}

\begin{remark}
If $p=2$, that is $N=2k$, then $R^{{\rm rad}}_{k, N} = \frac{(2k-1)^2 (2k-3)^2 \cdots 1^2}{4^k}$ which is same as the optimal constant of the geometric type Rellich inequality by Owen \cite{O}. 
\end{remark}

\begin{remark}
It is already known that $R_{2, \gamma}=R^{{\rm rad}}_{2, \gamma}$ and $R_{2, \gamma}$ is not attained for $\gamma = p,N$ for $p=2$ (see \cite{CM(2012)}),  $R_{3, 2}=R^{{\rm rad}}_{3, 2}$ and $R_{3, 2}$ is not attained (see \cite{N}). We conjecture that for $\gamma = p, N$, $R_{k, \gamma}=R^{{\rm rad}}_{k, \gamma}$ also hold for any $k$ and $N$. 
\end{remark}

As a corollary of Theorem \ref{Main thm}, we also obtain a non-sharp ($a>1$) critical Rellich inequality for radially symmetric functions as follows.

\begin{cor}\label{cor main}
Let $N > k \ge 2, m \in \N, p=\frac{N}{k}$ and $a \ge 1$.
Then the inequality
\begin{align}\label{NSCR p}
R_{k,p}^{{\rm rad}} \int_{B_R} \frac{|u|^p}{|x|^{N} \( \log \frac{aR}{|x|} \)^p} \,dx \le \int_{B_R} |\nabla^k u|^p\,dx
\end{align}
holds for any $u \in W_{0, {\rm rad}}^{k, p}(B_R)$, where $R_{k,p}^{{\rm rad}}$ is given in Theorem \ref{Main thm}. Moreover the optimal constant of (\ref{NSCR p}) is $R_{k,p}^{{\rm rad}}$ which is indenpendent of $a \ge 1$, and is not attained.
\end{cor}

\begin{remark}\label{rem AS}(Optimal constant on $W_{0, {\rm rad}}^{k,2} (k \ge 3)$)
Let $a \gg 1$. For simplicity, we consider the case where $k=2m, m \ge 2$ and $N=4m$. 
In Theorem 2.3.(a) in \cite{ASant}, the authors showed that the inequality
\begin{align}\label{AS thm}
A(N,m)^2 \int_{B_1} \frac{|u|^2}{|x|^{4m} \( \log \frac{a}{|x|} \)^2} \,dx \le \int_{B_1} |\lap^m u|^2\,dx
\end{align}
holds for any $u \in W_{0, {\rm rad}}^{2m, 2}(B_1)$, where the constant $A(N,m)$ is given by 
\begin{align*}
A(N,m) = \frac{N}{4} \( \frac{1}{2^{2m-2}} \prod_{i=0}^{m-2} (4i +2) (8m -4i -6) \),
\end{align*}
and $A(N,m)$ is the optimal constant of (\ref{AS thm}). But, unfortunately, only the argument for showing the optimality of $A(N,m)$ in \cite{ASant} contains a gap. In fact, they used the following radial test function $\psi_\delta$.
\begin{align*}
\psi_\delta (x) = \( \log \frac{a}{|x|} \)^{A(N,m) \( \prod_{i=1}^m (N-2i) \)^{-1} -\delta} \varphi (x),
\end{align*}
where $\varphi \in C_c^{\infty}(B_1)$ is a radial function, where $\varphi \equiv 1$ on $B_{1/2}$ and $\varphi \equiv 0$ on $B_1 \setminus B_{3/4}$. 
For example, if we consider the case where $m=3$, that is, $N=12$, then we have
\begin{align*}
A(N,m) \( \prod_{i=1}^m (N-2i) \)^{-1} = \frac{189}{160} > \frac{2-1}{2}
\end{align*}
which implies that $| \psi_\delta |_{W_0^{4,2}} = \infty$ for any small $\delta >0$.  Therefore it seems that $\psi_\delta$ is unsuitable as a test function since $\psi_\delta \not\in W_{0,{\rm rad}}^{6,2}(B_1)$. See also \S \ref{gap}. 

The optimal constant of (\ref{AS thm}) is $R^{{\rm rad}}_{2m,2}$ in Theorem \ref{Main thm}.
\end{remark}

A few comments are in order. 
Our minimization problem $R^{{\rm rad}}_{2m,\gamma}$ is related to the following polyhoarmonic elliptic equation with Dirichlet boundary conditions in the critical dimensions $N=4m$:
\begin{align}\label{EL}
\begin{cases}
(-\lap)^m u = R^{{\rm rad}}_{2m,\gamma} \,\frac{u}{|x|^N (\log \frac{R}{|x|})^{\gamma}} \quad &\text{in} \,\, B_R \subset \re^N, \\
\qquad D^{\, \beta} u = 0  \,\,( \,|\,\beta| \le m-1 )  &\text{on} \,\, \pd B_R.
\end{cases}
\end{align}
The minimizer for $R^{{\rm rad}}_{2m,\gamma}$ is a ground state radial solution of the Euler-Lagrange equation (\ref{EL}). 

By \cite{DHA, AGS, ASant, CM(2012)}, the critical Rellich inequalities have been studied so far based on ``a reduction of the dimension argument'' which is also called Brezis-Vazquez transformation or Maz'ya  transformation (ref. Corollary 3. in Section 2.1.7 in \cite{Mazya}, \cite{BV}). 
In contrast with it, our argument is based on the argument by Davies-Hinz \cite{DH} for the subcritical Rellich inequalities. 
This argument is more direct and calculations are simpler than it especially in the higher order case.

Since the function $|x|^{-N} \( \log \frac{aR}{|x|} \)^{-\gamma}$ is radially decreasing on $B_R$ if $a \ge e^{\frac{\gamma}{N}}$. Therefore, we might expect to obtain a non-sharp critical Rellich inequality:
\begin{align}\label{NSCR}
R^{{\rm rad}}_{k, \gamma} \int_{B_R} \frac{|u|^p}{|x|^N \( \log \frac{aR}{|x|} \)^\gamma} \,dx \le \int_{B_R} |\nabla^k u|^p\,dx
\end{align}
at least for $u \in W_\vartheta^{k,p} (B_R) := \biggr\{ u \in W^{k,p} (\Omega) \,\,\biggr|\,\, \lap^j u_{|_{\pd \Omega}} = 0$ in the sense of traces for $j \in \left[ 0, \frac{k}{2} \right)\,\, \biggr\}$ 
by Theorem \ref{Main thm} and the iterated Talenti comparison principle (ref. Theorem 1 in \cite{T(CP)} and Proposition 3 in \cite{GGS}). 
Unfortunately this does not follows from our result directly since we use the assumption $u'(R)=0$ for radial function $u$ to show our result, see also \cite{GGM, GGS}. However we believe that this might be a just technical reason and the non-sharp critical Rellich inequality (\ref{NSCR}) holds for any $u \in W_0^{k,p}(B_R)$. 



This paper is organized as follows:
In \S \ref{limiting form}, we consider {\it a limit} of the second order subcritical Rellich inequality (\ref{Rel}) as $p \nearrow \frac{N}{2}$ via some transformation. 
Of course, we can not consider a limit of the subcritical Rellich inequality (\ref{Rel}) in the usual sense. However, we can consider a limit of an inequality which is equivalent to the subcritical Rellich inequality (\ref{Rel}) via the transformation. 
This observation is inspired by a paper \cite{I} which is a study for the Hardy and the  Sobolev inequalities. 
For some limits of the Hardy and the Sobolev inequalities, see a survey \cite{S(RIMS)}. 
From this observation, we can catch a glimpse of the strategy of the proof of the second order critical Rellich inequality. More precisely, we observe that some Hardy type inequalities are important ingredients to show the critical Rellich inequality. 
In order to show the higher order critical Rellich inequalities in Theorem \ref{Main thm}, we need more general Hardy type inequalities than it. 
Therefore, in \S \ref{S new Hardy}, we study such Hardy type inequalities and their improvements. 
To the best of our knowledge, these Hardy type inequalities are new. 
In \S \ref{critical Rellich}, we show Theorem \ref{Main thm} by dividing two parts. 
In Part I, we derive the critical Rellich inequalities from the Hardy type inequalities in \S \ref{S new Hardy} and the subcritical Rellich type inequalities. And also, we show the attainability of $R_{k,\gamma}^{{\rm rad}}$ for $\gamma \in (p,N)$. 
In Part II, we calculate the explicit values of $R_{k,\gamma}^{{\rm rad}}$ for $\gamma =p, N$ and show the non-attainability of $R_{k,\gamma}^{{\rm rad}}$ for $\gamma =p, N$. In order to calculate them, logarithmic type functions $\( \log \frac{R}{|x|} \)^\alpha \,(\alpha \not= 1)$ are important. For the Trudinger-Moser inequality, it is known that the Moser type function $\( \log \frac{R}{|x|} \)$ is important to calculate the optimal exponents even in the higher order case (ref. \cite{M, A}). 
%
In \S \ref{gap}, we explain the cause of the gap in \cite{ASant} for deriving the optimal constant $R_{k,p}^{{\text rad}}$. 
In \S \ref{Appendix}, we calculate $\nabla^k \( \log \frac{R}{|x|} \)^\alpha$
and show the (non-)compactness of the related embeddings 
of our minimization problems $R^{{\rm rad}}_{k, \gamma}$. 

We fix several notations: 
$\w_{N-1}$ denotes an area of the unit sphere $\mathbb{S}^{N-1}$ in $\re^N$. 
$X_{{\rm rad}} = \{ \, u \in X \, | \, u \,\,\text{is radial} \, \}$.
Throughout the paper, if a radial function $u$ is written as $u(x) = \tilde{u}(|x|)$ by some function $\tilde{u} = \tilde{u}(r)$, we write $u(x)= u(|x|)$ with admitting some ambiguity.

%
%

\section{A limit of the subcritical 2nd-order Rellich inequality and an observation for the critical Rellich inequality}\label{limiting form}

Recently, in a inspiring paper \cite{I} the following improved Hardy type inequality is founded. 
\begin{align}\label{IH_p}
\( \frac{N-p}{p} \)^p \int_{B_R} \frac{|u|^p}{|x|^p \( 1- \( \frac{|x|}{R} \)^{\frac{N-p}{p-1}} \)^p} \,dx  \le \int_{B_R} \left| \nabla u \cdot \frac{x}{|x|} \right|^p \,dx
\end{align}
Actually, the inequality (\ref{IH_p}) is equivalent to the classical Hardy inequality (\ref{H_p}) on $\re^N$ based on the transformation: 
\begin{align}\label{Itrans}
u(x)= w(y),\,\, {\rm where}\,\, \( |x|^{-\frac{N-p}{p-1}} - R^{-\frac{N-p}{p-1}} \)\, \frac{x}{|x|} = |y|^{-\frac{N-p}{p-1}} \, \frac{y}{|y|}
\end{align}
One of the virtues of the improved Hardy type inequality (\ref{IH_p}) is that we can take a limit of (\ref{IH_p}) as $p \nearrow N$ in the usual sense. 
This is in striking contrast with the classical Hardy inequality (\ref{H_p}). 
As a limit of (\ref{IH_p}) as $p \nearrow N$, we can obtain the critical Hardy inequality (\ref{H_N}). 
In this sense, we can say that the critical Hardy inequality (\ref{H_N}) is a limiting form of the subcritical Hardy inequality (\ref{H_p}) on $\re^N$ as $p \nearrow N$ via the equivalent inequality (\ref{IH_p}) and the transformation (\ref{Itrans}). 

\begin{remark}\label{rem trans}(Various transformations)
Various transformations which connect between two derivative norms $\| \nabla (\cdot) \|_{L^p(\Omega)}$ are considered such as (\ref{Itrans}) by \cite{F, Z, HK, I, ST, S(JDE), S(ArXiv)}. 
Actually, we can also regard these transformations as special cases or generalized cases of a transformation in a paper \cite{F}. In fact, the following transformation for $u \in W_{0, {\rm rad}}^{1,2}(B_1)$ is given in Theorem 18. in \cite{F}, where $B_1, \Omega \subset \re^2$.
\begin{align}\label{Ftrans}
w(y)= u(x),\,\, {\rm where}\,\, G_{\Omega, z} (y) = G_{B_1, 0}(x) = -\frac{1}{2\pi} \log |x|
\end{align}
and $G_{\Omega, z}(y)$ is the Green function in a domain $\Omega$, which has a singularity at $z \in \Omega$. 
We can observe that the transformation (\ref{Itrans}) is (\ref{Ftrans}) in the case where $W_{0, {\rm rad}}^{1,p}(B_1), p < N$ and $z=0, B_1 \subset \re^N = \Omega$. 
An explanation of the other transformations, see e.g. \S 2 in \cite{S(ArXiv)}. 
\end{remark}

In this section, we consider an analog of the Hardy inequality (\ref{H_p}) for the subcritical second order Rellich inequality (\ref{Rel}), that is, 
we find {\it a limiting form} of the subcritical second order Rellich inequality (\ref{Rel}) as $p \nearrow \frac{N}{2}$ only for radial functions via some transformations. 
Differently from the first order case, there is no such beautiful transformation in the second order case. 
However, from Remark \ref{rem trans}, we can expect that the below transformation (\ref{trans}) with $\alpha = \frac{N-2p}{p-1}$ is suitable. 
More generally, we shall consider the transformation (\ref{trans}) with a general exponent $\alpha$ below.  
Thanks to the transformation (\ref{trans}), we can obtain a limiting form of the subcritical second order Rellich inequality (\ref{Rel}). A little strangely, it is a first order inequality, but it is a important ingredient to show the critical Rellich inequality associated with $R^{{\rm rad}}_{2,\frac{N}{2}}, R^{{\rm rad}}_{2, N}$,
see also \S \ref{critical Rellich}.

Consider the following transformation:
\begin{align}\label{trans}
u(r)=w(t), \,\text{where} \,\,r=f(t)= \( R^{-\alpha}+ t^{-\alpha} \)^{-\frac{1}{\alpha}} \,\, \text{i.e.}\,\, r^{-\alpha} -R^{-\alpha} = t^{-\alpha}, \,t \in (0,\infty)
\end{align}
Two functions $w \in C_{{\rm rad}}^1(\re^N \setminus \{ 0\}) \,\cap\, C(\re^N), \,u \in C_{{\rm rad}}^1(B_R \setminus \{ 0\}) \,\cap\, C(B_R)$ are radial functions. 
Let $\alpha >0$. Then we have the followings.
\begin{align*}
f'(t) &= \( \frac{r}{t} \)^{\alpha +1} \\
f''(t) &= (\alpha+1 ) \( \frac{r}{t^2} \)^{\alpha +1} \( r^\alpha - t^\alpha \)\\
w'(t) &= u'(r) f'(t) \\
w''(t) &= u''(r) f'(t)^2 + u'(r) f''(t) \\
\lap w(t) 
&= f'(t)^2 \left[ u''(r) + \frac{N-1}{r} u'(r) \left\{  \frac{\alpha +1}{N-1} + \frac{N- \alpha -2}{N-1}\frac{t^\alpha}{r^\alpha} \right\} \right]
\end{align*}
Since $\frac{t^\alpha}{r^\alpha} = \frac{1}{1-  r^\alpha R^{-\alpha}}$, we define the differential operator $L_{p, \alpha}$ as follows.
\begin{align}\label{L_p}
L_{p,\alpha} u &= u''(r) + \frac{N-1}{r} u'(r) \left\{  \frac{\alpha +1}{N-1} + \frac{N- \alpha -2}{N-1}\frac{1}{1- r^\alpha R^{-\alpha}} \right\} \\
&= \lap u(r) + \frac{u'(r)}{r} \,\frac{N-\alpha -2}{\( \frac{R}{r} \)^{\alpha} -1} \notag
\end{align}
Then we have $\lap w(t)= f'(t)^2 L_{p, \alpha} u(r)$ and  
\begin{align*}
\int_{\re^N} | \lap w|^p \,dy 
&= \w_{N-1} \int_0^\infty |L_{p, \alpha} u(r)|^p f'(t)^{2p} t^{N-1} \,dt \\
&= \w_{N-1} \int_0^R |L_{p, \alpha} u(r)|^p \( 1-r^\alpha R^{-\alpha} \)^{\frac{(2p-1)(\alpha +1) +1-N}{\alpha}} r^{N-1} \,dr \\
&= \int_{B_R} | L_{p, \alpha} u|^p \( 1- \( \frac{|x|}{R}\)^\alpha \)^{\,\beta} \,dx
\end{align*}
On the other hand, we have
\begin{align*}
\int_{\re^N} \frac{|w|^p}{|y|^{2p}} dy 
&= \w_{N-1} \int_0^\infty |u(r)|^p \, t^{N-1-2p} \,dt \\
&= \w_{N-1} \int_0^\infty |u(r)|^p \, t^{N-1-2p} \( f'(t) \)^{-1}\,dr 
=\int_{B_R} \frac{|u|^p}{|x|^{2p} \( 1- \( \frac{|x|}{R}\)^\alpha \)^{\frac{N-2p +\alpha}{\alpha}} } dx
\end{align*}

Consequently, we obtain the following.

\begin{prop}
Let $1< p< \frac{N}{2}, \alpha >0$ and $\beta = \frac{(2p-1)(\alpha +1) +1-N}{\alpha}$. Then the subcritical Rellich inequality on $\re^N$ for radial functions $w$:
\begin{equation}\label{R_p}
\( \frac{N (p-1) (N-2p)}{p^2} \)^p \int_{\re^N} \frac{|w|^p}{|y|^{2p}} dy
\le \int_{\re^N} | \lap w|^p \,dy 
\end{equation}
is equivalent to the following inequality for radial functions $u$:
\begin{equation}\label{IR_p}
\( \frac{N (p-1) (N-2p)}{p^2} \)^p \int_{B_R} \frac{|u|^p}{|x|^{2p} \( 1- \( \frac{|x|}{R}\)^\alpha \)^{\frac{N-2p +\alpha}{\alpha}} } dx
\le \int_{B_R} | L_{p, \alpha} u|^p \( 1- \( \frac{|x|}{R}\)^\alpha \)^{\,\beta} \,dx
\end{equation}
under the transformation (\ref{trans}). 
\end{prop}

\begin{remark}
If $\alpha = \frac{N-2p}{2p-1}$ and $p=1$, then we have
\begin{equation}\label{L to lap}
\int_{B_R} | L_{p, \alpha} u|^p \( 1- \( \frac{|x|}{R}\)^\alpha \)^{\,\beta} \,dx = \int_{B_R} | \lap u |^p \,dx.
\end{equation}
However, since the subcritical Rellich inequality (\ref{Rel}) does not hold for $p=1$, we exclude the case where $p=1$. Therefore, the equality (\ref{L to lap}) does not hold in general when $p>1$ .
\end{remark}

Now we take a limit of the inequality (\ref{IR_p}) as $p \nearrow \frac{N}{2}$, which is equivalent to the subcritical Rellich inequality (\ref{R_p}) on $\re^N$. If $\alpha \not\to 0$ as $p \nearrow \frac{N}{2}$, then the left-hand side of the inequality (\ref{IR_p}) becomes an indeterminate form as $p \nearrow \frac{N}{2}$. Therefore, in order to obtain a limiting form, we assume that $\alpha \to 0$ as $p \nearrow \frac{N}{2}$. Especially, we assume that $\lim_{p \nearrow \frac{N}{2}} \frac{N-2p}{\alpha} = A \in (0, \infty)$. 
Since $1-r^x \sim x \log \frac{1}{r} \,(x \to 0)$, the left-hand side of the inequality (\ref{IR_p}) is 
\begin{align*}
&\( \frac{N (p-1) (N-2p)}{p^2} \)^p \int_{B_R} \frac{|u|^p}{|x|^{2p} \( 1-\( \frac{|x|}{R}\)^{\alpha} \)^{\frac{N-2p}{\alpha} +1} } \,dx\\
&\sim \( \frac{2(N-2)}{N} \)^{\frac{N}{2}} (N-2p)^{\frac{N}{2}} \alpha^{-A-1} \int_{B_R} \frac{|u|^{\frac{N}{2}} }{ |x|^{N} \( \log \frac{R}{|x|} \)^{1+A}} \,dx \quad \( p \nearrow \frac{N}{2} \).
\end{align*}
On the other hand, in the right-hand side of the inequality (\ref{IR_p}), we have 
\begin{align*}
L_{p,\alpha} u &= u''(r) + \frac{N-1}{r} u'(r) \left\{  \frac{\alpha +1}{N-1} + \frac{N- \alpha -2}{N-1}\frac{1}{1- r^{\alpha} R^{-\alpha}} \right\} \\
&\sim \frac{(N-2)}{r \log \frac{R}{r}} u'(r) \, \alpha^{-1}  \quad \( p \nearrow \frac{N}{2} \),\\
 \(1- \( \frac{|x|}{R} \)^{\alpha} \)^{\,\beta} 
 &\sim \alpha^{N-1-A} \( \log \frac{R}{|x|} \)^{N-1-A} \quad \( p \nearrow \frac{N}{2} \).
\end{align*}
Therefore we observe that the limiting form of the inequality (\ref{IR_p}) as $p \nearrow \frac{N}{2}$ is 
\begin{align}\label{lim ineq}
\( \frac{2}{N} \,A \)^{\frac{N}{2}}  \int_{B_R} \frac{|u|^{\frac{N}{2}} }{ |x|^{N} \( \log \frac{R}{|x|} \)^{1+A}} \,dx 
\le \int_{B_R} \frac{|\nabla u|^{\frac{N}{2}} }{ |x|^{\frac{N}{2}} \( \log \frac{R}{|x|} \)^{A+1-\frac{N}{2}}} \,dx.
\end{align}

Note that the inequality (\ref{lim ineq}) is already known by \cite{MOW}. 
Consequently, we obtain the following.

\begin{prop}
We obtain the inequality (\ref{lim ineq}) as a limiting form of the subcritical Rellich inequality (\ref{R_p}) on $\re^N$ as $p \nearrow \frac{N}{2}$ via the equivalent inequality (\ref{IR_p}) and the transformation (\ref{trans}) with $\alpha$ which satisfies $\lim_{p \nearrow \frac{N}{2}} \frac{N-2p}{\alpha} = A \in (0, \infty)$.
\end{prop}

The inequality (\ref{lim ineq}) is an important ingredient to show the critical Rellich inequality. Indeed, if we can show the inequality:
\begin{align}\label{lim ineq right}
C \int_{B_R} \frac{|\nabla u|^{\frac{N}{2}} }{ |x|^{\frac{N}{2}} \( \log \frac{R}{|x|} \)^{A+1-\frac{N}{2}}} \,dx \le \int_{B_R} |\lap u|^{\frac{N}{2}} \,dx,
\end{align}
then we can obtain the desired 2nd-order critical Rellich inequality from (\ref{lim ineq}) and (\ref{lim ineq right}). 
Actually, the inequality (\ref{lim ineq right}) holds when $A=\frac{N}{2} -1$ or $N-1$, see \S \ref{S new Hardy} and \S \ref{critical Rellich}. 
In order to prove Theorem \ref{Main thm} completely, we need more general Hardy type inequalities than (\ref{lim ineq right}) in the next section. 

\begin{remark}
If we choose $\alpha= \frac{N-2p}{p-1}$, then $\frac{N-2p+\alpha}{\alpha} = p=\beta$ and $A=\frac{N}{2} -1$. On the other hand, if we choose $\alpha = \frac{N-2p}{2p-1}$, then $\frac{N-2p+\alpha}{\alpha} = 2p, \beta =0$ and $A=N-1$.
\end{remark}

%
%

\section{Another Hardy type inequality with two singularities at the origin and the boundary}\label{S new Hardy}

In Proposition 1.2. in \cite{HK} ($a \gg 1$ case) and Proposition 1 in \cite{II(2016)} ($a=1$ case), the following generalization of the critical Hardy inequality (\ref{H_N}) to  the weighted critical Sobolev spaces $W_0^{1,p}(B_R; |x|^{p-N} \,dx)$ is investigated, where $a \ge 1$.
\begin{align}\label{weight Hardy}
\( \frac{p-1}{p} \)^p \int_{B_R} \frac{|u|^p}{|x|^N \( \log \frac{aR}{|x|} \)^p} \,dx \le \int_{B_R} \left| \nabla u \cdot \frac{x}{|x|} \right|^p |x|^{p-N} \,dx
\end{align}

We observe that the inequality (\ref{weight Hardy}) goes to the critical Hardy inequality (\ref{H_N}) as $p \nearrow N$. In this section, we investigate another generalization of the critical Hardy inequality (\ref{H_N}) in the subcritical Sobolev spaces $W_0^{1,p}(B_R)$. Our inequality has similar structures to (\ref{weight Hardy}), see p.101 in \cite{II(2016)} and Remark \ref{rem virtual}.

\begin{theorem}\label{thm new Hardy}
Let $1 < p \le N$. Then the following inequality
\begin{align}\label{new Hardy}
\( \frac{p-1}{p} \)^p \int_{B_R} \frac{|u|^p}{|x|^p \( \log \frac{R}{|x|} \)^p} \,dx \le \int_{B_R} \left| \nabla u \cdot \frac{x}{|x|} \right|^p \,dx
\end{align}
holds for any $u \in W_0^{1,p} (B_R)$. And the constant $\( \frac{p-1}{p} \)^p$ is optimal and is not attained. 
Furthermore the following improved Hardy inequality
\begin{align}\label{new IHardy}
\( \frac{p-1}{p} \)^p \int_{B_R} \frac{|u|^p}{|x|^p \( \log \frac{R}{|x|} \)^p} \,dx + \phi_{N,p} (u) \le \int_{B_R} \left| \nabla u \cdot \frac{x}{|x|} \right|^p \,dx
\end{align}
holds for any $u \in W_0^{1,p} (B_R)$, where 
\begin{align*}
\phi_{N,p} (u ) = (N-p) \( \frac{p-1}{p} \)^{p-1}\int_{B_R} \frac{|u|^p}{|x|^p \( \log \frac{R}{|x|} \)^{p-1}} \,dx.
\end{align*}
\end{theorem}

\begin{remark}\label{rem virtual}
The inequality (\ref{new Hardy}) is invariant under the scaling $u_\la (x) = \la^{-\frac{p-1}{p}} u\( y\),$ where $y= \( \frac{|x|}{R} \)^{\la -1} x\, (x \in B_R)$. Furthermore, in the same way as the proof in \cite{II}, we can show that the radial function $\( \log \frac{R}{|x|} \)^{\frac{p-1}{p}}$ is the virtual minimizer of 
\begin{align}\label{min P}
\( \frac{p-1}{p} \)^p = \inf_{u \in W_0^{1,p}(B_R ) \setminus \{ 0\} } \dfrac{\int_{B_R} |\nabla u |^{p} \,dx}{ \int_{B_R} \frac{|u|^p}{|x|^p \( \log \frac{R}{|x|} \)^p} \,dx}.
\end{align}
More precisely, we can show as follows. If there exists a nonnegative minimizer of (\ref{min P}), then there also exists a nonnegative radial minimizer $U$. On the other hand, if there exist two nonnegative minimizers $u, v$, then there exists $C=C(u,v) >0$ such that $u =C v$. Applying this property for $u=U$ and $v = U_\la$ implies that $C=1$ and $U=U_\la$ thanks to the scale invariance structure. From this, we observe that $U=c \( \log \frac{R}{|x|} \)^{\frac{p-1}{p}} (c >0)$. However $U \not\in W_0^{1,p}(B_R)$ which is contradiction.
\end{remark}

Here, we recall the improved Hardy type inequality (\ref{IH_p}) in \S \ref{limiting form}, which is shown by Ioku \cite{I} . 
Since for any $x \in B_R$
\begin{align*}
\frac{p-1}{N-p} \left[ 1- \( \frac{|x|}{R} \)^{\frac{N-p}{p-1}} \right]  \le \log \frac{R}{|x|},
\end{align*}
we can see that our inequality (\ref{new Hardy}) is weaker than the improved inequality (\ref{IH_p}). However, both inequalities (\ref{IH_p}), (\ref{new Hardy}) go to  the critical Hardy inequality (\ref{H_N}) as $p \nearrow N$ and also have the scale invariance structure under each scaling. Besides, the proof of our inequality (\ref{new Hardy}) is simpler and more direct since we do not use some transformation like (\ref{Itrans}).

\begin{proof}
First, we show the inequality (\ref{new Hardy}) in the similar way to \cite{TF, MOW}.  Note that 
$$
{\rm div} \( \frac{x}{|x|^\alpha \( \log \frac{R}{|x|} \)^\beta} \)
= \frac{N-\alpha}{|x|^\alpha \( \log \frac{R}{|x|} \)^\beta}
+ \frac{\beta}{|x|^\alpha \( \log \frac{R}{|x|} \)^{\beta+1}}.
$$
Now we set $\alpha = p$ and $\beta = p-1$. Then we have 
\begin{align*}
&\int_{B_R} \frac{|u|^p}{|x|^p \( \log \frac{R}{|x|} \)^p} \,dx \\
&= \frac{1}{p-1} \int_{B_R} \left[ {\rm div} \( \frac{x}{|x|^p \( \log \frac{R}{|x|} \)^{p-1}}\) -\frac{N-p}{|x|^p \( \log \frac{R}{|x|} \)^{p-1}}  \right]  |u|^p \,dx \\
&= -\frac{p}{p-1} \int_{B_R} \frac{|u|^{p-2}u \,(\nabla u \cdot x)}{|x|^p \( \log \frac{R}{|x|} \)^{p-1}} \,dx -\frac{N-p}{p-1} \int_{B_R} \frac{|u|^p}{|x|^p \( \log \frac{R}{|x|} \)^{p-1}} \,dx\\
&\le \frac{p}{p-1} \( \int_{B_R} \left| \nabla u \cdot \frac{x}{|x|} \right|^p \,dx \)^{\frac{1}{p}} \( \int_{B_R} \frac{|u|^p}{|x|^p \( \log \frac{R}{|x|} \)^p} \,dx  \)^{1-\frac{1}{p}} -\frac{N-p}{p-1} \int_{B_R} \frac{|u|^p}{|x|^p \( \log \frac{R}{|x|} \)^{p-1}} \,dx
\end{align*}
which implies that for any $u \neq 0$
\begin{align}\label{p mae}
\frac{p-1}{p}\( \int_{B_R} \frac{|u|^p}{|x|^p \( \log \frac{R}{|x|} \)^p} \,dx  \)^{\frac{1}{p}}
\le  \( \int_{B_R} \left| \nabla u \cdot \frac{x}{|x|} \right|^p \,dx \)^{\frac{1}{p}}
-\frac{N-p}{p} \frac{\int_{B_R} \frac{|u|^p}{|x|^p \( \log \frac{R}{|x|} \)^{p-1}} \,dx}{\( \int_{B_R} \frac{|u|^p}{|x|^p \( \log \frac{R}{|x|} \)^p} \,dx  \)^{\frac{p-1}{p}}}.
\end{align}
Therefore we obtain the inequality (\ref{new Hardy}) for $p \in (1, N]$. 
Set 
\begin{align*}
A= \( \int_{B_R} \left| \nabla u \cdot \frac{x}{|x|} \right|^p \,dx \)^{\frac{1}{p}}, \quad
B= \frac{N-p}{p} \frac{\int_{B_R} \frac{|u|^p}{|x|^p \( \log \frac{R}{|x|} \)^{p-1}} \,dx}{\( \int_{B_R} \frac{|u|^p}{|x|^p \( \log \frac{R}{|x|} \)^p} \,dx  \)^{\frac{p-1}{p}}}.
\end{align*}
By the fundamental inequality: $(A-B)^p \le A^p -p (A-B)^{p-1} B\,(A \ge B)$ and the inequality: 
$A-B \ge \frac{p-1}{p}\( \int_{B_R} \frac{|u|^p}{|x|^p \( \log \frac{R}{|x|} \)^p} \,dx  \)^{\frac{1}{p}}$ from (\ref{p mae}), we have
\begin{align*}
\( \frac{p-1}{p} \)^p \int_{B_R} \frac{|u|^p}{|x|^p \( \log \frac{R}{|x|} \)^p} \,dx 
&\le A^p -p(A-B)^{p-1} B\\
&\le \int_{B_R} \left| \nabla u \cdot \frac{x}{|x|} \right|^p \,dx - \phi_{N,p} (u)
\end{align*}
which implies (\ref{new IHardy}) and the non-attainability of the optimal constant $\( \frac{p-1}{p} \)^p$ in (\ref{new Hardy}) except for $p=N$. Note that the case where $p=N$ is already shown by \cite{II}. Finally, we show the optimality of the constant $\( \frac{p-1}{p} \)^p$ in (\ref{new Hardy}). For $\gamma > \frac{p-1}{p}$, set
\begin{align*}
	\psi_\gamma (x)=
	\begin{cases}
		1, \,\,\,&\text{if} \,\,\,|x| \le \frac{R}{e},  \\
		\( \log \frac{R}{|x|} \)^\gamma, &\text{if} \,\,\, \frac{R}{e} \le |x| \le R.
	\end{cases}
\end{align*}
Then we have
\begin{align*}
\( \frac{p-1}{p} \)^p 
&\le \frac{\int_{B_R} \left| \nabla \psi_\gamma \cdot \frac{x}{|x|} \right|^p \,dx}{\int_{B_R} \frac{|\psi_\gamma|^p}{|x|^p \( \log \frac{R}{|x|} \)^p} \,dx} \\
&= \frac{\gamma^p \int_{R/e}^R \( \log \frac{R}{r} \)^{(\gamma -1)p} r^{N-p-1}\,dr}{\int_0^{R/e} \( \log \frac{R}{r} \)^{-p} r^{N-p-1}\,dr + \int_{R/e}^R \( \log \frac{R}{r} \)^{(\gamma -1)p} r^{N-p-1}\,dr} \\
&= \( \frac{p-1}{p} \)^p + o(1) \quad \( \beta \to \frac{p-1}{p} \).
\end{align*}
Therefore the constant $\( \frac{p-1}{p} \)^p$ in (\ref{new Hardy}) is optimal.
\end{proof} 

More generally, we can show the following inequality (\ref{GH pre}) which includes various inequalities in the same way as the above proof. The special case where $\alpha=N-p$ is shown by \cite{MOW}. We omit the proof.

\begin{theorem}\label{thm GH}
Let $1 < p < \infty$ and $\beta \ge 1-p$. Then the following inequality
\begin{align}\label{GH pre}
\( \frac{\beta+p -1}{p} \)^p \int_{B_R} \frac{|u|^p}{|x|^{\alpha +p} \( \log \frac{R}{|x|} \)^{\beta +p}} \,dx
+ \tilde{\psi}_{N,p,\alpha, \beta} (u) \le \int_{B_R} \frac{\left| \nabla u \cdot \frac{x}{|x|} \right|^p}{|x|^{\alpha} \( \log \frac{R}{|x|} \)^{\beta}}\,dx
\end{align}
holds for any $u \in C_c^1 (B_R)$ if $\alpha \le N -p$ and holds for any $u \in C_c^1 (B_R \setminus \{ 0 \})$ if $\alpha > N -p$, where 
\begin{align*}
\tilde{\psi}_{N,p, \alpha, \beta} (u ) = (N-p-\alpha) \( \frac{\beta +p-1}{p} \)^{p-1}\int_{B_R} \frac{|u|^p}{|x|^{p+\alpha} \( \log \frac{R}{|x|} \)^{\beta+ p-1}} \,dx.
\end{align*}
\end{theorem}

In the case where $\alpha = N-p$, the remainder term $\tilde{\psi}_{N,p,\alpha, \beta} (u)$ is zero. Therefore we give a remainder term of the inequality (\ref{GH pre}) only in this case.

\begin{theorem}\label{thm IGH}
Let $1 < p < \infty$ and $\beta \ge 1-p$. Then the inequality
\begin{align}\label{IGH}
\( \frac{\beta+p -1}{p} \)^p \int_{B_R} \frac{|u|^p}{|x|^{N} \( \log \frac{R}{|x|} \)^{\beta +p}} \,dx
+ \phi_{N,p, \beta} (u) \le \int_{B_R} \frac{\left| \nabla u \cdot \frac{x}{|x|} \right|^p}{|x|^{N-p} \( \log \frac{R}{|x|} \)^{\beta}}\,dx
\end{align}
holds for any $u \in C_c^1 (B_R)$, where $C_{p,N}$ depends only on $p$ and $N$, and 
\begin{align*}
\phi_{N,p, \beta} (u ) =
\begin{cases}
C_{p,N} \int_{B_R} \frac{\( \log \frac{R}{|x|} \)^{p-1}}{|x|^{N-p}} \left| \nabla \( \frac{u(x)}{\( \log \frac{R}{|x|} \)^{\frac{\beta +p-1}{p}}} \) \cdot \frac{x}{|x|} \right|^p \,dx \quad \text{if}\,\, p \in [2, \infty) \\
C_{p,N} \( \int_{B_R} \frac{\( \log \frac{R}{|x|} \)^{p-1}}{|x|^{N-p}} \left| \nabla \( \frac{u(x)}{\( \log \frac{R}{|x|} \)^{\frac{\beta +p-1}{p}}} \) \cdot \frac{x}{|x|} \right|^p \,dx \)^{\frac{2}{p}} 
\( \int_{B_R} \frac{\( \log \frac{R}{|x|} \)^{-\beta}}{|x|^{N-p}} \left| \nabla u \cdot \frac{x}{|x|} \right|^p \,dx \)^{-\frac{2-p}{p}} \\
\hspace{17em}\text{if}\,\, p \in (1, 2) 
\end{cases}
\end{align*}
\end{theorem}

The proof of Theorem \ref{thm IGH} is same as it of Theorem 1 in \cite{II(2016)}, which is the case where $\beta = 0$. In their proof, they used the following transformation for $u$.
$$
v(x)= \( \log \frac{R}{|x|} \)^{-\frac{p-1}{p}} u(x)
$$
In order to show Theorem \ref{thm IGH}, it is enough to change it to the following transformation.
$$
v(x)= \( \log \frac{R}{|x|} \)^{-\frac{\beta +p-1}{p}} u(x)
$$
Note that $v(0)=0$ without $u(0)=0$. 
We omit the proof of Theorem \ref{thm IGH}. 

From Theorem \ref{thm GH} and Theorem \ref{thm IGH}, we obtain the following.

\begin{cor}\label{cor IGH}
Let $1 < p < \infty, \alpha \le N -p$ and $\beta \ge 1-p$. Then the following inequality
\begin{align}\label{GH}
\( \frac{\beta+p -1}{p} \)^p \int_{B_R} \frac{|u|^p}{|x|^{\alpha +p} \( \log \frac{R}{|x|} \)^{\beta +p}} \,dx
+ \psi_{N,p,\alpha, \beta} (u) \le \int_{B_R} \frac{\left| \nabla u \cdot \frac{x}{|x|} \right|^p}{|x|^{\alpha} \( \log \frac{R}{|x|} \)^{\beta}}\,dx
\end{align}
holds for any $u \in C_c^1 (B_R)$, where 
\begin{align*}
\psi_{N,p, \alpha, \beta} (u) = 
\begin{cases}
\tilde{\psi}_{N,p, \alpha, \beta} (u ) \quad &\text{if}\,\, \alpha < N-p,\\
\phi_{N,p, \beta} (u ) &\text{if}\,\, \alpha = N-p.
\end{cases}
\end{align*}
\end{cor}

From Theorem \ref{thm GH}, we have the following inequality.

\begin{theorem}\label{thm lap Hardy}
Let $1 < p < \infty$ and $\beta \ge 1-p$. Then the following inequality
\begin{align}\label{lap Hardy}
\( \frac{\beta+p -1}{p} \)^p \int_{B_R} \frac{|\nabla u|^p}{|x|^{\alpha +p} \( \log \frac{R}{|x|} \)^{\beta +p}} \,dx
\le \int_{B_R} \frac{|\lap u|^p}{|x|^{\alpha} \( \log \frac{R}{|x|} \)^{\beta}}\,dx
\end{align}
holds for any radial functions $u \in C_{c, {\rm rad}}^2 (B_R)$. 
Especially, if $p=2, \alpha =0$ and $\beta = 0$, then the inequality 
\begin{align}\label{lap Hardy2}
\frac{1}{4} \int_{B_R} \frac{|\nabla u|^2}{|x|^2 \( \log \frac{R}{|x|} \)^2} \,dx 
+ \frac{N-2}{2} \int_{B_R} \frac{|\nabla u|^2}{|x|^2 \( \log \frac{R}{|x|} \)} \,dx
\le \int_{B_R} |\lap u |^2 \,dx 
\end{align}
holds for any functions $u \in C_c^2 (B_R)$.
\end{theorem}


\begin{proof}(Proof of Theorem \ref{thm lap Hardy})
First we assume that $u$ is a radial function. By Theorem \ref{thm GH}, 
we have
\begin{align*}
&\int_{B_R} \frac{|\lap u |^p}{|x|^{\alpha} \( \log \frac{R}{|x|} \)^{\beta}} \,dx \\
&= \w_{N-1} \int_0^R |u'' + (N-1) r^{-1} u'|^p \,r^{-\alpha +N-1} \( \log \frac{R}{r}\)^{-\beta}\,dr \\
&= \w_{N-1} \int_0^R | (r^{N-1} u' )' |^p \,r^{-\alpha -(N-1)p + N-1}\( \log \frac{R}{r} \)^{-\beta}  \,dr \\
&\ge \( \frac{\beta+ p-1}{p} \)^p \w_{N-1} \int_0^R |r^{N-1} u'|^p \,r^{-\alpha -Np + N-1} \( \log \frac{R}{r}\)^{-\beta-p} \,dr + \psi_{N,p, \alpha + (N-1)p, \beta} (r^{N-1} u')\\
&=\( \frac{\beta+ p-1}{p} \)^p \int_{B_R} \frac{|\nabla u|^p}{|x|^{\alpha + p} \( \log \frac{R}{|x|} \)^{\beta +p}} \,dx + (N-\alpha -Np) \( \frac{\beta+ p-1}{p} \)^{p-1} \int_{B_R} \frac{|\nabla u|^p}{|x|^{p+\alpha} \( \log \frac{R}{|x|} \)^{\beta +p-1}} \,dx.
\end{align*}
Next we assume $p=2, \alpha =0$ and $\beta = 0$. By Theorem \ref{thm new Hardy}, we can show (\ref{lap Hardy2}) without radially symmetry as follows.
\begin{align*}
\int_{B_R} |\lap u |^2 \,dx 
&=  \sum_{i=1}^N \int_{B_R} |\nabla u_{x_i} |^2\,dx \\
&\ge \( \frac{1}{2} \)^2  \sum_{i=1}^N \int_{B_R} \frac{|u_{x_i}|^2}{|x|^{2} \( \log \frac{R}{|x|} \)^{2}} \,dx + \phi_{N,2} (u_{x_i})\\
&= \frac{1}{4} \int_{B_R} \frac{|\nabla u|^2}{|x|^2 \( \log \frac{R}{|x|} \)^2} \,dx 
+ \frac{N-2}{2} \int_{B_R} \frac{|\nabla u|^2}{|x|^2 \( \log \frac{R}{|x|} \)} \,dx
\end{align*}


\end{proof}

%
%

\section{Critical Rellich inequalities: proof of Theorem \ref{Main thm}}\label{critical Rellich}

In this section, we show Theorem \ref{Main thm}. In Part I, we show the positivity and the attainability of $R_{k,\gamma}^{{\rm rad}}$. In Part II, we give the explicit values and the non-attainability of the optimal constant $R_{k,\gamma}^{{\rm rad}}$ for $\gamma =p, N$. 

\subsection{Inequality: Part I on the proof of Theorem \ref{Main thm}}\label{Inequality}

In this subsection, we show the lower bounds of $R_{k,p}^{{\rm rad}}$ and $R_{k,N}^{{\rm rad}}$ which implies that $R_{k,\gamma}^{{\rm rad}} > 0$ for $\gamma \in [p, N]$. Especially, we show 
\begin{align}\label{lower est}
R^{{\rm rad}}_{k, p}\ge 
\begin{cases}
\( \frac{N-k}{kN} \prod_{j=1}^m 2j \,(N-2j)\, \)^p &\text{if}\,\, k=2m, \\
\( \frac{N-k}{N} \prod_{j=1}^m 2j \,(N-2j) \,\)^p &\text{if}\,\, k=2m+1,
\end{cases}\,\,\,\,
R^{{\rm rad}}_{k, N}\ge \( \prod_{j=1}^k \frac{jN-k}{N} \)^p.
\end{align} 

More generally, we show the followings.
\begin{theorem}\label{gene Main}
(I) If $\alpha \le N-2mp$, then the following inequality holds for any radial functions $u \in C_c^{2m}(B_R)$.
\begin{align}\label{gene e bdy}
\int_{B_R} \frac{|\lap^m u|^p}{|x|^\alpha}\,dx 
\ge \( \prod_{j=1}^{2m} \frac{jp -1}{p} \)^p &\int_{B_R} \frac{|u|^p}{|x|^{\alpha +2mp} \( \log \frac{R}{|x|} \)^{2mp}} \notag \\
&+ \( \prod_{j=1}^{2m-1} \frac{jp -1}{p} \)^p \psi_{N,p, \alpha + (2m-1)p, (2m-1)p} (u),
\end{align}
where $\psi_{N,p, \alpha, \beta} (u)$ is given by Corollary \ref{cor IGH}. 

\noindent
(II) If $2(1-p) < \alpha \le N-2mp$, then the following inequality holds for any radial functions $u \in C_c^{2m}(B_R)$. 
\begin{align}
\label{gene e ori}
\int_{B_R} \frac{|\lap^m u|^p}{|x|^\alpha}\,dx 
\ge &D(N,m,p,\alpha)^p \int_{B_R} \frac{|u|^p}{|x|^{\alpha +2mp} \( \log \frac{R}{|x|} \)^{p}} \notag\\
&+ \( \frac{Np-N+\alpha +2(m-1)p}{p \,C(N, m-1,p,2mp+\alpha)} \)^p \psi_{N, p, \alpha+(2m -1)p, 0} (u),
\end{align}
where $C(N,m,p,\beta)$ is given by Theorem \ref{DH lemma} and $D(N,m,p,\alpha)$ is given by 
\begin{align*}
D(N,m,p,\alpha)
= \left[ \prod_{j=1}^{m-1} \frac{(2pj + N-2mp -\alpha) \left\{ p(N-2-2j) -N +2mp+\alpha \right\}}{p^2}  \right] \\
\cdot \frac{(p-1)\left\{ (N-2)p -N +2mp +\alpha \right\} }{p^2}.
\end{align*}
(III) If $\alpha \le N-(2m+1)p$, then the following inequality holds for any radial functions $u \in C_c^{2m+1}(B_R)$.
\begin{align}\label{gene o bdy}
\int_{B_R} \frac{|\nabla \lap^m u|^p}{|x|^\alpha}\,dx 
\ge \( \prod_{j=1}^{2m+1} \frac{jp -1}{p} \)^p &\int_{B_R} \frac{|u|^p}{|x|^{\alpha +(2m+1)p} \( \log \frac{R}{|x|} \)^{(2m+1)p}}\notag \\
&+ \( \prod_{j=1}^{2m} \frac{jp -1}{p} \)^p \psi_{N,p, \alpha + 2mp, 2mp} (u).
\end{align}
(IV) If $2-3p < \alpha \le N-(2m+1)p$, then the following inequality holds for any radial functions $u \in C_c^{2m+1}(B_R)$. 
\begin{align}
\label{gene o ori}
\int_{B_R} \frac{|\nabla \lap^m u|^p}{|x|^\alpha}\,dx 
&\ge E(N,m,p,\alpha)^p \int_{B_R} \frac{|u|^p}{|x|^{\alpha +(2m+1)p} \( \log \frac{R}{|x|} \)^{p}}\notag\\
&+ \( \frac{(Np-N+\alpha +2mp-p)(N-\alpha-p)}{p^2 \,C(N, m-1, p, (2m+1)p +\alpha)} \)^p \psi_{N, p, \alpha+2mp, 0} (u)
\end{align}
Here the constant $E(N,m,p,\alpha)$ is given by 
\begin{align*}
E(N,m,p,\alpha)= D(N,m,p,\alpha+p) \( \frac{N-\alpha -p}{p} \).
\end{align*}
\end{theorem}

In order to show Theorem \ref{gene Main}, we recall that the subcritical Rellich type inequalities by Davies-Hinz \cite{DH}, an inequality by Musina \cite{M(2014)} and the following Hardy type inequality for any functions $u \in C_c^1(B_R)$, where $p < N+\delta$. 
\begin{align}\label{H 1 to 0}
\int_{B_R} |x|^\delta |\nabla u|^p \,dx 
\ge \( \frac{N+\delta}{p} -1 \)^p \int_{B_R} |x|^{\delta -p} | u|^p\,dx
\end{align}

\begin{theorem}{(\cite{DH} Theorem 12)}\label{DH lemma}

\noindent
Let $m \in \N, 2 \{ 1 +(m-1) p \} < \beta <N$. Then the following inequality holds for any functions $u \in C_c^{2m}(B_R \setminus \{ 0\})$.
\begin{align}\label{DH even}
&\int_{B_R} \frac{|u|^p}{|x|^{\beta}} \, dx \leq C(N, m, p, \beta )^p \int_{B_R} \frac{|\lap^m u|^p}{|x|^{\beta -2mp}} \, dx \\
&\text{where} \,\, C(N, m, p, \beta ) 
= \prod_{k=0}^{m-1} \frac{p^2}{(N-\beta +2kp) \{ (p-1) (N-2) +\beta -2(1+kp) \} }. \nonumber
\end{align}

\end{theorem}

\begin{theorem}{(\cite{M(2014)} Theorem 1.2)}\label{M lemma}
Let $\delta \in \re, p>1$ and let $m \ge 1$ be a given integer. Then the inequality
\begin{align}\label{M 2 to 1}
\int_{B_R} |x|^\delta |\lap u|^p \,dx 
\ge \left| N -\frac{N+\delta}{p} \right|^p \int_{B_R} |x|^{\delta -p} |\nabla u|^p\,dx
\end{align}
holds for any radial functions $u \in C_{c, {\rm rad}}^{2}(B_R \setminus \{ 0\})$. 
\end{theorem}

For the inequality (\ref{M 2 to 1}) with $p=2$ for any functions, see Theorem 1.7. in \cite{TZ}. 

\begin{lemma}\label{lemma density}
Theorem \ref{DH lemma} holds true even if $C_c^{2m}(B_R \setminus \{ 0\})$ is replaced by $C_c^{2m}(B_R )$. 
Furthermore, Theorem \ref{M lemma} holds true even if $C_{c, {\rm rad}}^{2}(B_R \setminus \{ 0\})$ is replaced by $C_{c, {\rm rad}}^{2}(B_R )$. 
\end{lemma}

Before the proof of Lemma \ref{lemma density}, we recall the one dimensional Hardy type inequality (\ref{1dim Hardy}) for any $a \in \re, p>1$ and $w \in C^1(0,R)$ with $w(0) =w(R) =0$:
\begin{align}\label{1dim Hardy}
\int_0^R r^a | w'|^p \,dr \ge \left| \frac{a+1-p}{p} \right|^p \int_0^R r^{a-p} | w|^p \,dr
\end{align}

\begin{proof}(Proof of Lemma \ref{lemma density})\\
Let $\eta$ be a smooth function with $\eta \equiv 0$ on $B_{1/2}$ and $\eta \equiv 1$ on $\re^N \setminus B_{1}$. For small $\ep >0$, we set $\eta_\ep (x)= \eta (\ep^{-1} x)$. Then we see that $\eta_\ep u \in C_c^{2m} (B_R \setminus \{ 0\})$ for $u \in C_c^{2m}(B_R)$ and $|\lap^m \eta_\ep | \le C \ep^{-2m}$. Then we have
\begin{align*}
\int_{B_R} \frac{|\lap^m (\eta_\ep u -u)|^p}{|x|^{\beta -2mp}} \, dx 
&\le C (u, \nabla u, \cdots, \lap^{2m} u) \, \ep^{-2mp} \int_{B_\ep} \frac{1}{|x|^{\beta -2mp}} \, dx \\
&\le C \ep^{-\beta +N} \to 0 \quad (\ep \to 0)
\end{align*}
since $\beta < N$. Therefore we have
\begin{align*}
\int_{B_R} \frac{|u|^p}{|x|^{\beta}} \, dx 
&=\lim_{\ep \to 0} \int_{B_R} \frac{|\eta_\ep u|^p}{|x|^{\beta}} \, dx \\
&\leq C(N, m, p, \beta )^p \lim_{\ep \to 0} \int_{B_R} \frac{|\lap^m (\eta_\ep u)|^p}{|x|^{\beta -2mp}} \, dx \\
&= C(N, m, p, \beta )^p \int_{B_R} \frac{|\lap^m u|^p}{|x|^{\beta -2mp}} \, dx
\end{align*}
for any functions $u \in C_c^{2m}(B_R)$. Hence Theorem \ref{DH lemma} holds true even if $C_c^{2m}(B_R \setminus \{ 0\})$ is replaced by $C_c^{2m}(B_R )$. 

Next we recall the proof of Theorem \ref{M lemma} and check that there are no problems. 
Let $u \in C_{c, {\rm rad}}^{2}(B_R)$. Then we have
\begin{align*}
\int_{B_R} |x|^\delta |\lap u|^p \,dx 
&= \w_{N-1} \int_0^R \left| u'' + \frac{N-1}{r} u' \right|^p r^{N-1-\delta}\, dr\\
&= \w_{N-1} \int_0^R \left| ( r^{N-1} u' )' \right|^p r^{-(N-1)p+ N-1-\delta}\, dr
\end{align*}
Here, note that $r^{N-1} u' (r) =0$ at $r=0$ without $u(0)=0$. From (\ref{1dim Hardy}), we have
\begin{align*}
\int_{B_R} |x|^\delta |\lap u|^p \,dx
&\ge \left| \frac{(p-1)N +\delta}{p} \right|^p \int_0^R \left| r^{N-1} u'  \right|^p r^{-Np+ N-1-\delta}\, dr\\
&= \left| N -\frac{N+\delta}{p} \right|^p \int_{B_R} |x|^{\delta -p} |\nabla u|^p\,dx.
\end{align*}
Theorem \ref{M lemma} holds true even if $C_{c, {\rm rad}}^{2}(B_R \setminus \{ 0\})$ is replaced by $C_{c, {\rm rad}}^{2}(B_R )$. 
\end{proof}

From Lemma \ref{lemma density} and several Hardy type inequalities (\ref{GH}), (\ref{lap Hardy}), (\ref{H 1 to 0}), we can obtain Theorem \ref{gene Main}. 

\begin{proof}(Proof of Theorem \ref{gene Main})\\
(I) Let $u \in C_{c, {\rm rad}}^{2m}(B_R)$. By using the inequalities (\ref{GH}), (\ref{lap Hardy}) $2m$ times totally, we obtain (\ref{gene e ori}) as follows.
\begin{align*}
\int_{B_R} \frac{|\lap^m u|^p}{|x|^{\alpha}}\,dx 
\ge &\( \frac{(p-1)(2p-1)(3p-1) \cdots (2mp-1)}{p^{2m}} \)^p \int_{B_R} \frac{|u|^p}{|x|^{\alpha+ 2mp} \( \log \frac{R}{|x|} \)^{2mp}} \,dx \\
&+\( \prod_{j=1}^{2m-1} \frac{jp -1}{p} \)^p \psi_{N,p, \alpha + (2m-1)p, (2m-1)p}(u).
\end{align*}

\noindent
(II) Let $u \in C_{c, {\rm rad}}^{2m}(B_R)$. By the inequality (\ref{DH even}) with $\beta = \alpha + 2(m-1)p$, we have
\begin{align}\label{(II)1}
\int_{B_R} \frac{|\lap^m u|^p}{|x|^\alpha}\,dx \ge C(N,m-1, p, 2(m-1)p+\alpha)^{-p} \int_{B_R} \frac{|\lap u|^p}{|x|^{\alpha+ 2(m-1)p}}\,dx. 
\end{align}
Applying the inequality (\ref{M 2 to 1}) with $\delta = -\alpha- 2(m-1)p$ implies that 
\begin{align}\label{(II)2}
\int_{B_R} \frac{|\lap u|^p}{|x|^{\alpha+ 2(m-1)p}}\,dx
\ge \( N- \frac{N-\alpha -2(m-1)p}{p} \)^p \int_{B_R} \frac{|\nabla u|^p}{|x|^{\alpha+ 2mp-p}}\,dx.
\end{align}
Finally, by the inequality (\ref{GH}), 
we have
\begin{align}\label{(II)3}
\int_{B_R} \frac{|\nabla u|^p}{|x|^{\alpha+ 2mp-p}}\,dx 
\ge \( \frac{p-1}{p} \)^p \int_{B_R} \frac{|u|^p}{|x|^{\alpha+ 2mp} \( \log \frac{R}{|x|} \)^p} \,dx + \psi_{N,p, \alpha + (2m-1)p, 0} (u).
\end{align}
Therefore we obtain (\ref{gene e ori}) from (\ref{(II)1}), (\ref{(II)2}) and (\ref{(II)3}).

\noindent
(III) The proof of (\ref{gene o bdy}) is completely same as it of (I). 

\noindent
(IV) Let $u \in C_{c, {\rm rad}}^{2m+1}(B_R)$. By the inequality (\ref{H 1 to 0}),
\begin{align}\label{(IV)1}
\int_{B_R} \frac{|\nabla \lap^m u|^p}{|x|^\alpha}\,dx \ge \( \frac{N-\alpha}{p} -1\)^p \int_{B_R} \frac{|\lap^m u|^p}{|x|^{\alpha+ p}}\,dx. 
\end{align}
Applying the inequality (\ref{gene e ori}) implies that 
\begin{align}\label{(IV)2}
\int_{B_R} \frac{|\lap^m u|^p}{|x|^{\alpha+ p}}\,dx
&\ge D(N, m, p, \alpha +p)^p \int_{B_R} \frac{|u|^p}{|x|^{\alpha+ (2m+1)p} \( \log \frac{R}{|x|} \)^p} \,dx \notag\\
&+ \( \frac{Np-N+\alpha +2mp-p}{p \,C(N, m-1, p, (2m+1)p+\alpha)} \)^p \psi_{N, p, \alpha+2mp, 0} (u)
\end{align}
Therefore we obtain (\ref{gene o ori}) from (\ref{(IV)1}) and (\ref{(IV)2}).
\end{proof}

\begin{proof}(Part I on the proof of Theorem \ref{Main thm})\\
We can obtain the lower estimates (\ref{lower est}) 
from Theorem \ref{gene Main} with $\alpha = 0, p=\frac{N}{2m}$. Since $R_{k,p}^{{\rm rad}}, R_{k,N}^{{\rm rad}} > 0$, we have $R_{k,\gamma}^{{\rm rad}} > 0$ for $\gamma \in [p, N]$. Moreover $R_{k,\gamma}^{{\rm rad}}$ is attained for $\gamma \in (p,N)$ from the compactness of the embedding in Proposition \ref{prop cpt} in \S \ref{Appendix}. 
\end{proof}

%
%

\subsection{Optimality and attainability: Part II on the proof of Theorem \ref{Main thm}}\label{optimality}

Let $kp =N$. 
In order to calculate the optimal constant $R^{{\rm rad}}_{k,\gamma}$, it is important to find a  virtual minimizer of $R^{{\rm rad}}_{k,\gamma}$. Differently from the first order case, it seems difficult to find a scale invariance structure of the derivative term $|u|_{W_0^{k,p}}$ even if we assume that $u$ is a radial function. In this point of view, it seems difficult to find the virtual minimizer of $R^{{\rm rad}}_{k,\gamma}$.  
However we expect the existence of such important functions which play similar roles to the first order case. Such important functions are 
$$
V_1(x)= \( \log \frac{R}{|x|} \)^{\frac{p-1}{p}} = \( \log \frac{R}{|x|} \)^{\frac{N-k}{N}} \,\,{\rm for}\,\,\, \gamma = p 
$$ 
and 
$$
V_2(x)= \( \log \frac{R}{|x|} \)^{\frac{N-1}{p}}= \( \log \frac{R}{|x|} \)^{k \frac{N-1}{N}}\,\,{\rm for}\,\, \,\gamma =N.
$$
Note that $\gamma =p$ is the optimal exponent with respect to the singularity of the potential $|x|^{-N} \( \log \frac{R}{|x|} \)^{-\gamma}$ at the origin. 
On the other hand, $\gamma =N$ is the optimal exponent with respect to the boundary singularity of the potential $|x|^{-N} \( \log \frac{R}{|x|} \)^{-\gamma}$. 
Set $V_{1,\ep} (x)= \( \log \frac{R}{|x|} \)^{\frac{p-1}{p} -\ep}$ and $V_{2,\ep} (x)= \( \log \frac{R}{|x|} \)^{\frac{N-1}{p} + \ep}$ for small $\ep >0$. For fixed $0< \delta \ll 1$, we see that
\begin{align*}
&\int_{B_\delta} |\nabla^k V_{1,\ep} |^p \,dx < \infty \,\, \text{for each}\,\, \ep >0, \text{however}\,\,
\int_{B_\delta} |\nabla^k V_{1,\ep} |^p \,dx \nearrow \infty \,(\ep \to 0), \\
&\int_{B_R \setminus B_{R-\delta}} |\nabla^k V_{2,\ep} |^p \,dx < \infty \,\, \text{for each}\,\, \ep >0, \text{however}\,\,
\int_{B_R \setminus B_{R-\delta}} |\nabla^k V_{2,\ep} |^p \,dx \nearrow \infty\,(\ep \to 0).
\end{align*}
In the end of Part II on the proof of Theorem \ref{Main thm}, we shall show $V_1$ (respectively, $V_2$) is a virtual minimizer of $R_{k,p}^{{\rm rad}}$ (respectively, $R_{k,N}^{{\rm rad}}$). 
For the details, see the proof below. 

\begin{remark}
We observe that in the first order case $k=1$, $V_1(x)= V_2(x)= \( \log \frac{R}{|x|} \)^{\frac{N-1}{N}}$ which is known as a virtual minimizer of the critical Hardy inequality (\ref{H_N}). Except for the first order case, $V_1 \not= V_2$. In this paper, we treat only the higher order case $k \in \N, k \ge 2$. However, even in the fractional case where $0 < k < 1$, we believe that these two functions $V_1, V_2$ are important.  
\end{remark}

\begin{proof}(Part II on the proof of Theorem \ref{Main thm})\\
First we consider a radial test function $\phi_\ep \in W_{0, {\rm rad}}^{k,p}(B_R)$ which is given by
\begin{align*}
\phi_\ep (x) = V_{1,\ep}(x) \, \varphi (x) = \( \log \frac{R}{|x|} \)^{\frac{p-1}{p} -\ep} \varphi (x)
\end{align*}
where $\varphi \in C_c^{\infty}(B_R)$ is a radial function, where $\varphi \equiv 1$ on $B_{R/2}$ and $\varphi \equiv 0$ on $B_R \setminus B_{3R/4}$. From Proposition \ref{prop log} in \S \ref{Appendix}, we have
\begin{align*}
|\nabla^k V_{1,\ep}(x)| =
|E_k| \( \frac{p-1}{p} -\ep  \) |x|^{-k} \( \log \frac{R}{|x|} \)^{-\frac{1}{p} -\ep}
+ o\( \( \log \frac{R}{|x|} \)^{-\frac{1}{p} -\ep}  \) \,\,(|x| \to 0)
\end{align*}
where $E_k = C_{m, 2m-1}$ if $k=2m$, and $E_k = D_{m, 2m}$ if $k=2m+1$, $C_{m,2m-1}$ and $D_{m, 2m}$ are given by Proposition \ref{prop log} in \S \ref{Appendix}. 
Therefore we have
\begin{align*}
R^{{\rm rad}}_{k, p} 
&\le \dfrac{|\phi_\ep |^p_{k,p}}{ \int_{B_R} \frac{|\phi_\ep|^p}{|x|^N \( \log \frac{R}{|x|} \)^p} \,dx} \\
&=\dfrac{|E_k|^p \( \frac{p-1}{p} -\ep  \)^p \w_{N-1} \int_0^{\frac{R}{2}} \( \log\frac{R}{r} \)^{-1-p\ep} \frac{dr}{r} + o\( \int_0^{\frac{R}{2}} \( \log\frac{R}{r} \)^{-1-p\ep} \frac{dr}{r} \)}{\w_{N-1} \int_0^{\frac{R}{2}} \( \log\frac{R}{r} \)^{-1-p\ep} \frac{dr}{r} + o\( \int_0^{\frac{R}{2}} \( \log\frac{R}{r} \)^{-1-p\ep} \frac{dr}{r} \)} \\
&= |E_k|^p \( \frac{p-1}{p}  \)^p + o(1) \quad (\ep \to 0).
\end{align*}
Since 
\begin{align*}
|E_k|^p \( \frac{p-1}{p}  \)^p = 
\begin{cases}
\( \frac{N-k}{kN} \prod_{j=1}^m 2j \,(N-2j)\, \)^p &\text{if}\,\, k=2m, \\
\( \frac{N-k}{N} \prod_{j=1}^m 2j \,(N-2j) \,\)^p &\text{if}\,\, k=2m+1,
\end{cases}
\end{align*}
we obtain 
\begin{align*}
R^{{\rm rad}}_{k, p} = 
\begin{cases}
\( \frac{N-k}{kN} \prod_{j=1}^m 2j \,(N-2j)\, \)^p &\text{if}\,\, k=2m, \\
\( \frac{N-k}{N} \prod_{j=1}^m 2j \,(N-2j) \,\)^p &\text{if}\,\, k=2m+1
\end{cases}
\end{align*}
from Part I on the proof of Theorem \ref{Main thm}. 
Next we consider a radial test function $\psi_\ep \in W_{0, {\rm rad}}^{k,p}(B_R)$ which is given by
\begin{align*}
\psi_\ep (x) = V_{2,\ep}(x) \, \varphi (x) = \( \log \frac{R}{|x|} \)^{\frac{N-1}{p} +\ep} \( 1-\varphi (x) \).
\end{align*}
From Proposition \ref{prop log} in \S \ref{Appendix}, we have
\begin{align*}
|\nabla^k V_{2,\ep}(x)| =
\prod_{i=0}^{k-1} \( \frac{N-1}{p} +\ep -i \) |x|^{-k} \( \log \frac{R}{|x|} \)^{-\frac{1}{p} +\ep}
+ o\( \( \log \frac{R}{|x|} \)^{-\frac{1}{p} + \ep}  \) \,\,(|x| \to R).
\end{align*}
Therefore we have
\begin{align*}
R^{{\rm rad}}_{k, N} 
&\le \dfrac{|\psi_\ep |^p_{k,p} }{ \int_{B_R} \frac{|\psi_\ep|^p}{|x|^N \( \log \frac{R}{|x|} \)^N} \,dx} \\
&=\dfrac{ \( \prod_{i=0}^{k-1} \( \frac{N-1}{p} +\ep -i \) \)^p \w_{N-1} \int_{\frac{3R}{4}}^R \( \log\frac{R}{r} \)^{-1+p\ep} \frac{dr}{r} + o\( \int_{\frac{3R}{4}}^R \( \log\frac{R}{r} \)^{-1+p\ep} \frac{dr}{r} \)}{\w_{N-1} \int_{\frac{3R}{4}}^R \( \log\frac{R}{r} \)^{-1+p\ep} \frac{dr}{r} + o\( \int_{\frac{3R}{4}}^R \( \log\frac{R}{r} \)^{-1+p\ep} \frac{dr}{r} \)} \\
&= \( \prod_{i=0}^{k-1} \( \frac{N-1}{p} -i \)  \)^p+ o(1) \quad (\ep \to 0).
\end{align*}
Since 
\begin{align*}
\( \prod_{i=0}^{k-1} \( \frac{N-1}{p} -i \)  \)^p 
= \( \prod_{i=0}^{k-1} \( \frac{(k-i)N -k}{N} \)  \)^p 
=\( \prod_{j=1}^k \frac{jN-k}{N} \)^p
\end{align*}
we obtain
\begin{align*}
R^{{\rm rad}}_{k, N} = \( \prod_{j=1}^k \frac{jN-k}{N} \)^p
\end{align*}
from Part I on the proof of Theorem \ref{Main thm}. 
Moreover, by using the same test functions $\phi_\ep, \psi_\ep$, we can show that $R^{{\rm rad}}_{k, \gamma} = 0$ if $\gamma \not\in [p,N]$. 

Finally, we shall show that $R^{{\rm rad}}_{k, p}, R^{{\rm rad}}_{k, N}$ are not attained. 
Assume that $R^{{\rm rad}}_{k, p}$ is attained by $u \in W_{0, {\rm rad}}^{k,p} \setminus \{ 0\}$.
Then $\psi_{N,p,N-p, 0}(u)=0$ in Theorem \ref{gene Main} (II), (IV) which implies that $u(x)=c \( \log \frac{R}{|x|} \)^{\frac{p-1}{p}} =c V_1(x) \,(c \not= 0) \not\in W_{0, {\rm rad}}^{k,p} (B_R)$. This is a contradiction. On the other hand, if we assume that $R^{{\rm rad}}_{k, N}$ is attained by $u \in W_{0, {\rm rad}}^{k,p} \setminus \{ 0\}$, 
then $\psi_{N,p,N-p, N-p}(u)=0$ in Theorem \ref{gene Main} (I), (III) which implies that $u(x)=c \( \log \frac{R}{|x|} \)^{\frac{N-1}{p}} =c V_2(x) \,(c \not= 0) \not\in W_{0, {\rm rad}}^{k,p} (B_R)$. This is also a contradiction. 
Hence $R^{{\rm rad}}_{k, p}, R^{{\rm rad}}_{k, N}$ are not attained. 
The proof of Theorem \ref{Main thm} is now complete. 
\end{proof}


\begin{proof}(Proof of Corollary \ref{cor main})\\
Since $\log \frac{R}{|x|} \le \log \frac{aR}{|x|}$ for any $a \ge 1$ and any $x \in B_R$, 
the inequality (\ref{NSCR p}) immediately follows from Theorem \ref{Main thm} (iii). 
In order to show the optimality of the constant $R_{k,p}^{{\rm rad}}$ in (\ref{NSCR p}), it is enough to change the test function $\phi_\ep (x) = \( \log \frac{R}{|x|} \)^{\frac{p-1}{p} -\ep} \varphi (x)$ which is in Part II on the proof of Theorem \ref{Main thm} to $\( \log \frac{aR}{|x|} \)^{\frac{p-1}{p} -\ep} \varphi (x)$. 
Finally, the non-attainability of the optimal constant $R_{k,p}^{{\rm rad}}$ in (\ref{NSCR p}) follows from the non-attainability of $R_{k,p}^{{\rm rad}}$ in Theorem \ref{Main thm}. 
\end{proof}

\section{The cause of the gap in \cite{ASant}}\label{gap}

In Remark \ref{rem AS}, we explained the gap of the optimality of the constant $A(N,m)$ in the higher order critical Rellich inequality. On the other hand, we can obtain the optimal constant $R_{k,p}^{{\text{rad}}}$ in Theorem \ref{Main thm} correctly in our argument. Where does the gap come from?  
Actually, our argument resembles the argument in \cite{ASant} in the view of tools, which are three Hardy-Rellich type inequalities, for showing the higher order critical Rellich inequality. The only difference between our argument and it in \cite{ASant} is the order of use of these three tools. 
In this section, we explain concretely where the gap comes from when $(k,p)=(4, 2)$, that is $N=8$. 

We recall the argument in \cite{ASant} to show the higher order critical Rellich inequality. They used the following three Hardy-Rellich type inequalities:
\begin{align}\label{2 to 1}
\int_{B_1} | \lap w |^2 \,dx &\ge \frac{N^2}{4} \int_{B_1} \frac{|\nabla w|^2}{|x|^2}\,dx\\
\label{R_p}
\int_{B_1} \frac{| \lap \mathbf{v} |^2}{|x|^2} \,dx &\ge \( \frac{(N-6)(N+2)}{4} \)^2 \int_{B_1} \frac{|\mathbf{v}|^2}{|x|^6}\,dx\\
\label{log H}
\int_{B_1} \frac{| \nabla f|^2}{|x|^6}\,dx &\ge \frac{1}{4} \int_{B_1} \frac{| f|^2}{|x|^8 (\log \frac{1}{|x|})^2}\,dx
\end{align}
In fact, we can derive the 8th order critical Rellich inequality by three Hardy-Rellich type inequalities (\ref{2 to 1}), (\ref{R_p}), (\ref{log H}) as follows.
\begin{align*}
\int_{B_1} | \lap^2 u |^2 \,dx 
&\ge \frac{N^2}{4} \int_{B_1} \frac{|\nabla \lap u|^2}{|x|^2}\,dx \\
&= \frac{N^2}{4} \int_{B_1} \frac{|\lap (\nabla u)|^2}{|x|^2}\,dx \\
&\ge \frac{N^2}{4} \( \frac{(N-6)(N+2)}{4} \)^2 \int_{B_1} \frac{| \nabla u|^2}{|x|^6}\,dx \\
&\ge \frac{N^2}{4} \( \frac{(N-6)(N+2)}{4} \)^2 \frac{1}{4} \int_{B_1} \frac{| u|^2}{|x|^8 (\log \frac{1}{|x|})^2}\,dx
\end{align*}
On the other hand, our aurgument is as follows, see the proof of Theorem \ref{gene Main} (II).
\begin{align*}
\int_{B_1} | \lap^2 u |^2 \,dx 
&\ge \( \frac{(N-4)(N-2+4-2)}{2^2} \)^2 \int_{B_1} \frac{|\lap u|^2}{|x|^4}\,dx \\
&\ge \( \frac{N(N-4)}{4} \)^2 \( N -\frac{N-4}{2} \)^2  \int_{B_1} \frac{| \nabla u|^2}{|x|^6}\,dx \\
&\ge \( \frac{N(N-4)}{4} \)^2 \( N -\frac{N-4}{2} \)^2 \frac{1}{4} \int_{B_1} \frac{| u|^2}{|x|^8 (\log \frac{1}{|x|})^2}\,dx
\end{align*}
If we substitute $8$ for $N$, then we have
\begin{align*}
A(8, 2)&= \frac{N^2}{4} \( \frac{(N-6)(N+2)}{4} \)^2 \frac{1}{4} \\
&=100 \not= 16 \cdot 36\\
&= \( \frac{N(N-4)}{4} \)^2 \( N -\frac{N-4}{2} \)^2 \frac{1}{4} = R_{4,2}^{{\text rad}}.
\end{align*}
Recently, the authors in \cite{HT} showed that the optimal constant in (\ref{R_p}) can be improved for curl-free vector fields, see Corollary 4. in \cite{HT}. 
Therefore we can observe that $\( \frac{(N-6)(N+2)}{4} \)^2$ in (\ref{R_p}) is not the optimal constant for curl-free vector fields $\mathbf{v}=\nabla u$. More precisely, the authors in \cite{HT} obtained the optimal constant $77$ as follows.
\begin{align}\label{HT Cor}
\int_{B_1} \frac{|\nabla \lap u|^2}{|x|^2}\,dx  
\ge 77 \int_{B_1} \frac{| \nabla u|^2}{|x|^6}\,dx,\,\,\text{where}\,\, B_1 \subset \re^8
\end{align}
Besides, since $u$ is a radial function, we can improve (\ref{HT Cor}) a little bit more. In fact, by (\ref{H 1 to 0}) and (\ref{M 2 to 1}), we have 
\begin{align*}
\( \frac{N-4}{2} \)^{-2} \int_{B_1} \frac{|\nabla \lap u|^2}{|x|^2}\,dx  
\ge
 \int_{B_1} \frac{| \lap u|^2}{|x|^4}\,dx
\ge
\( \frac{N+4}{2} \)^2 \int_{B_1} \frac{| \nabla u|^2}{|x|^6}\,dx,\text{where}\,\, B_1 \subset \re^N,
\end{align*}
which implies that 
\begin{align}\label{HT Cor rad}
\int_{B_1} \frac{|\nabla \lap u|^2}{|x|^2}\,dx  
\ge
144 \int_{B_1} \frac{| \nabla u|^2}{|x|^6}\,dx,\,\,\text{where}\,\, B_1 \subset \re^8.
\end{align}
Therefore, if we use (\ref{HT Cor rad}) instead of (\ref{R_p}), then we can obtain the optimal constant $R_{4,2}^{{\text{rad}}}$ correctly even in the argument in  \cite{ASant}. 

As a consequence, the cause of the gap in \cite{ASant} comes from the non-optimality of the constant in the inequality (\ref{R_p}) for curl-free radial vector fields $\mathbf{v}=\nabla u$.

%
%

\section{Appendix}\label{Appendix}

\begin{prop}\label{prop cpt}
Let $1 < p =\frac{N}{k}, k \ge 2$ and $P_\gamma (x) = |x|^{-N} \( \log \frac{R}{|x|} \)^{-\gamma}$. Then the embedding: $W_{0, {\rm rad}}^{k,p}(B_R) \hookrightarrow L^p \( B_R; P_\gamma (x) \,dx \)$ is compact for $\gamma \in (p,N)$ and is non-compact for $\gamma = p,N$. 
\end{prop}

\begin{proof}
First we assume that $\gamma \in (p,N)$. Let $(u_m)_{m=1}^\infty \subset W_{0, {\rm rad}}^{k,p}(B_R)$ be a bounded sequence. Then there exists a subsequence $(u_{m_k})_{k=1}^\infty$ such that 
\begin{align}\label{L^r}
&u_{m_k} \rightharpoonup u \,\, \text{in} \,\, W_{0, {\rm rad}}^{k,p}(B_R), \nonumber \\
&u_{m_k} \to u \,\, \text{in} \,\, L^r(B_R ) \quad \text{for any} \,\, r \in (1, \infty )
\end{align}
see e.g. Theorem 2.1 and Theorem 2.4 in \cite{G book}. For any small $\ep >0$, there exists $\delta >0$ such that 
\begin{align}\label{log est}
\( \log \frac{R}{|x|} \)^{p-\gamma} < \ep \,\,\text{for}\,\, x \in B_\delta\,\,\text{and}\,\,
\( \log \frac{R}{|x|} \)^{N-\gamma} < \ep \,\,\text{for}\,\, x \in B_R \setminus B_{R-\delta}.
\end{align}
Form (\ref{L^r}) and (\ref{log est}), we have
\begin{align*}
\int_{B_R} \frac{|u_{m_k}-u|^q}{ |x|^N (\log \frac{aR}{|x|} )^{\gamma}} dx
&\le \ep \int_{B_\delta } \frac{|u_{m_k}-u|^p}{ |x|^N ( \log \frac{R}{|x|} )^{p}} dx + C_\delta \,\| u_{m_k} -u \|^p_{L^p(B_R )} +
\ep \int_{B_\delta } \frac{|u_{m_k}-u|^p}{ |x|^N ( \log \frac{R}{|x|} )^{p}} dx \\
&\le 2\ep C\, | u_{m_k} -u |_{k,p}^p + C \| u_{m_k} -u \|^p_{L^p(B_R )}\\
&\le C \ep + C \,C \| u_{m_k} -u \|^p_{L^p(B_R )} \to 0 \quad \text{as} \,\, \ep \to 0, k \to \infty.
\end{align*}
Thus the continuous embedding $W_{0, {\rm rad}}^{k,p}(B_R) \hookrightarrow L^p \( B_R; P_\gamma (x) \,dx \)$ is compact for $\gamma \in (p,N)$. 
On the other hand, we observe that the continuous embedding $W_{0, {\rm rad}}^{k,p}(B_R) \hookrightarrow L^p \( B_R; P_\gamma (x) \,dx \)$ is non-compact for $\gamma = p, N$ from the non-attainability of $R^{{\rm rad}}_{k,p}, R^{{\rm rad}}_{k,N}$ in Theorem \ref{Main thm}. 
Here we give a non-compact sequence of $W_{0, {\rm rad}}^{2,p}(B_R) \hookrightarrow L^p \( B_R; P_\gamma (x) \,dx \)$ for $\gamma = p, N (= 2p)$ concretely. 
Let $u \in C_c^{\infty}(B_R)$ be a radial function. Consider the scaling: $u_\la (r)=\la^a u(s)$, where $s=s(r) =r^\la R^{1-\la}$ for $\la >0$. Then we have
\begin{align}\label{log term}
\int_{B_R} \frac{|u_\la |^p}{|x|^N (\log \frac{R}{|x|})^\gamma} \,dx
= \la^{ap +\gamma -1} \int_{B_R} \frac{|u_\la |^p}{|y|^N (\log \frac{R}{|y|})^\gamma} \,dy.
\end{align}
And also, we see that
\begin{align}\label{lap term}
\int_{B_R} |\lap u_\la |^p \,dx 
&= \la^{ap} \w_{N-1} \int_0^R \left| \frac{d^2}{dr^2} u(s) + \frac{N-1}{r} \frac{d}{dr} u(s) \right|^p r^{N-1} \,dr \notag \\
&=\la^{ap} \w_{N-1} \int_0^R \left| u''(s) + \left\{ \frac{s''(r)}{\( s'(r) \)^2} + \frac{N-1}{r s'(r)} \right\} u'(s)  \right|^p \( s'(r) \, r \)^{N-1} \,ds \notag \\
&=\la^{ap +N-1} \w_{N-1} \int_0^R \left| u''(s) + \( \frac{N-2}{\la} +1 \) \frac{u'(s)}{s} \right|^p s^{N-1} \,ds \notag \\
&\le C \max \{ \la^{ap+N-1}, \,\la^{ap+p-1} \}. 
\end{align}
If $\gamma =p$, then we take $a=- \frac{p-1}{p}$ and $\la =\frac{1}{m} \,(m \in \N)$. By (\ref{lap term}), we see that $\{ u_{\frac{1}{m}} \}_{m=1}^\infty \subset W_{0, {\rm rad}}^{2,p}(B_R)$ is a bounded sequence which satisfies $u_{\frac{1}{m}} \rightharpoonup 0$ in $W_{0, {\rm rad}}^{2,p}(B_R)$ as $m \to \infty$. However $u_{\frac{1}{m}} \not\to 0$ in $L^p(B_R; P_p (x) \,dx)$ as $m \to \infty$ from (\ref{log term}). On the other hand, if $\gamma =N$, then we take $a=- \frac{N-1}{p}$ and $\la =m$. Then we also see that $u_m \rightharpoonup 0$ in $W_{0, {\rm rad}}^{2,p}(B_R)$ and $u_m \not\to 0$ in $L^p(B_R; P_N (x) \,dx)$ as $m \to \infty$. Hence the embedding: $W_{0, {\rm rad}}^{2,p}(B_R) \hookrightarrow L^p(B_R; P_\gamma (x) \,dx)$ is not compact if $\gamma =p$ or $\gamma =N$.
\end{proof}


\begin{prop}\label{prop log}
Let $m \in \N$. 
\begin{align}\label{log e}
&\lap^m \left[ \( \log \frac{R}{|x|} \)^\alpha \, \right] 
=  \sum_{j=0}^{2m -1} C_{m, j} \left\{ \prod_{i=0}^{2m -j -1}  (\alpha - i) \right\} 
|x|^{-2m} \( \log \frac{R}{|x|} \)^{\alpha -2m +j} \\
\label{log o}
&\nabla \lap^m \left[ \( \log \frac{R}{|x|} \)^\alpha \, \right] 
=  \sum_{j=0}^{2m} D_{m, j} \left\{ \prod_{i=0}^{2m -j}  (\alpha - i) \right\} 
|x|^{-2m-2} x \( \log \frac{R}{|x|} \)^{\alpha -2m +j -1} 
\end{align}
where $C_{m,j}$ and $D_{m,j}$ depend on $m, N$ and $j$, and satisfy as follows.
\begin{align*}
&C_{m,0} = D_{m,0}= 1\,\,(m \ge 1),\\
&C_{m, 2m -1} = \frac{(-1)^{m-1}}{2m} \( \prod_{j=1}^m 2j \,(N-2j) \) \,\,(m \ge 1),\\
&C_{m, j} =
\begin{cases}
C_{m-1, 1} + (N+2-4m) \,C_{m-1, 0} \quad &\text{if}\,\,j=1,\\
C_{m-1, j} + (N+2-4m) \,C_{m-1, j-1} -2(m-1)(N-2m) \,C_{m-1, j-2}  &\text{if}\,\,2 \le j \le 2m -3, \\
(N+2-4m) \,C_{m-1, 2m-3} -2(m-1)(N-2m) \,C_{m-1, 2m-4}  &\text{if}\,\, j =2m-2,
\end{cases}\\
&\hspace{3em}(m \ge 2)\\
&D_{m,j}= C_{m,j}-2m \,C_{m, j-1} \,\,( m \ge 1, \,1 \le j \le 2m -1),\\
&D_{m,2m} = -2m \,C_{m, 2m-1}\,\,(m \ge 1)
\end{align*}
\end{prop}

\begin{proof}
We have
\begin{align}\label{grad log}
&\nabla \left[ |x|^{-A} \( \log \frac{R}{|x|} \)^\alpha \right]
= |x|^{-A-2} x \left[ \alpha \( \log \frac{R}{|x|} \)^{\alpha -1} -A \( \log \frac{R}{|x|} \)^\alpha \right]\\
&\lap \left[ |x|^{-A} \( \log \frac{R}{|x|} \)^\alpha \right] 
= |x|^{-A-2} \,\biggr[ \alpha (\alpha -1) \( \log \frac{R}{|x|} \)^{\alpha -2}  \notag \\
&\hspace{7em}+\alpha (N-2 -2A) \( \log \frac{R}{|x|} \)^{\alpha -1} -A (N-2 -A) \( \log \frac{R}{|x|} \)^\alpha \,\biggr]\notag
\end{align}
Therefore we see that (\ref{log e}) holds when $m=1$. Now we assume that (\ref{log e}) holds for $m$. Then we have
\begin{align*}
&\lap^{m+1} \left[ \( \log \frac{R}{|x|} \)^\alpha \right] 
= \sum_{j=0}^{2m -1} C_{m, j} \left\{ \prod_{i=0}^{2m -j -1}  (\alpha - i) \right\} 
\lap \left[ |x|^{-2m} \( \log \frac{R}{|x|} \)^{\alpha -2m +j} \right] \\
&= \sum_{j=0}^{2m -1} C_{m, j} \left\{ \prod_{i=0}^{2m -j -1}  (\alpha - i) \right\} |x|^{-2(m+1)} \biggr[ \,(\alpha -2m +j) (\alpha -2m +j-1) \( \log \frac{R}{|x|} \)^{\alpha -2m +j-2} \\
&+(\alpha -2m +j) (N-2 -4m) \( \log \frac{R}{|x|} \)^{\alpha -2m +j-1}
-2m (N-2 -2m) \( \log \frac{R}{|x|} \)^{\alpha -2m +j}  \,\biggr] \\
&= |x|^{-2(m+1)} \Biggr[ \left\{ \prod_{i=0}^{2m+1}  (\alpha - i) \right\}  \( \log \frac{R}{|x|} \)^{\alpha -2m -2} \\
&+ \left\{ \prod_{i=0}^{2m}  (\alpha - i) \right\} \left\{  C_{m,1} + (N-2-4m) C_{m,0} \right\} \( \log \frac{R}{|x|} \)^{\alpha -2m -1}\\
&+ \sum_{k=2}^{2m-1} \left\{ \prod_{i=0}^{2m-k+1}  (\alpha - i) \right\} \left\{  C_{m,k} + (N-2-4m) C_{m,k-1} -2m\, (N-2-2m) C_{m,k-2} \right\} \( \log \frac{R}{|x|} \)^{\alpha -2m +k-2} \\
&+\left\{ \prod_{i=0}^{1}  (\alpha - i) \right\} \left\{  (N-2-4m) C_{m,2m-1} -2m\, (N-2-2m) C_{m,2m-2} \right\} \( \log \frac{R}{|x|} \)^{\alpha -2}\\
&-2 m \,(N-2 -2m) C_{m, 2m-1} \alpha \( \log \frac{R}{|x|} \)^{\alpha -1}  \,\Biggr]\\
&=\sum_{j=0}^{2(m+1) -1} C_{m+1, j} \left\{ \prod_{i=0}^{2(m+1) -j -1}  (\alpha - i) \right\} 
|x|^{-2(m+1)} \( \log \frac{R}{|x|} \)^{\alpha -2(m+1) +j}
\end{align*}
which implies that (\ref{log e}) holds for $m+1$. Thus  (\ref{log e}) holds for any $m \in \N$. From (\ref{log e}) and (\ref{grad log}), we can show (\ref{log o}). We omit the proof.
\end{proof}


\section*{Acknowledgment}
The author of this paper thanks to Prof. V. Nguyen and Prof. F. Takahashi for their useful information in terms of the main theorem and \S \ref{gap}. 
This work was (partly) supported by Osaka City University Advanced
Mathematical Institute (MEXT Joint Usage/Research Center on Mathematics
and Theoretical Physics). 
And also, the author was supported by JSPS KAKENHI Early-Career Scientists, No. JP19K14568.


\end{document}